\documentclass[12pt]{article}

\usepackage{amsmath}
\usepackage{amssymb}
\usepackage{amsthm}

\usepackage{hyperref}
\hypersetup{
	colorlinks = true,
	linkcolor = {blue},
	citecolor={blue}
}
\usepackage{xcolor}
\usepackage{graphicx}
\usepackage[color= orange!33!white, bordercolor = orange!20!black]{todonotes}
\usepackage{subcaption}
\usepackage{hyperref}
\usepackage{dsfont}

\usepackage{tikz}
\usetikzlibrary{math}

\theoremstyle{plain}
\newtheorem{teo}{Theorem}[section]
\newtheorem{prop}[teo]{Proposition}
\newtheorem{prob}{Problem}

\theoremstyle{definition}
\newtheorem{defin}[teo]{Definition}
\newtheorem*{conj*}{Conjecture}
\newtheorem*{rem}{Remark}

\newcommand{\N}{\mathbb{N}}            % Natural numbers
\newcommand{\R}{\mathbb{R}}             % Real numbers
\newcommand{\G}{\mathcal{G}}            % Graph
\DeclareMathOperator{\diag}{diag}

\title{New results and open problems on subgraph centrality \protect}
\author{Nikita Deniskin\footnote{Scuola Normale Superiore, Pisa, Italy. E--mail : {\tt nikita.deniskin@sns.it, michele.benzi@sns.it}.} \hspace{0.5cm} Michele Benzi\footnotemark[1]\protect}
\date{12 November 2021}

\begin{document}
	\maketitle

	\begin{abstract}
	\noindent Subgraph centrality, introduced by Estrada and Rodr\'iguez-Vel\'azquez in \cite{ERVsubgraphdefinition},
	has become a widely used centrality measure in the analysis of networks, with applications in biology,
	neuroscience, economics and many other fields. It is also worthy of study from a strictly mathematical 
	point of view, in view of its connections to topics in spectral graph theory, number theory, analytic matrix
	functions, and combinatorics. In this paper
	we present some new results and a list of open questions about subgraph centrality and other node centrality measures based on graph walks. 
	\end{abstract}
	\textit{AMS 2020 subject classifications:} 05C50, 15A16.
	
	\noindent \textit{Keywords:} Cospectral Vertices, Graph Walks, Subgraph Centrality, Interlacing.

	\section{Introduction}
	
	The ranking of nodes in a network has long been one of the most fundamental problems in Network Science. 
	Determining the most influential individuals in a social network (whether human or animal), discovering the essential proteins
	inside the cell, finding the most relevant web pages providing information on a given search topic or locating the best connected 
	hubs in a transportation network are clearly tasks of the utmost importance in many situations. In simple terms, a node centrality measure 
	is a function defined on the vertex set of  a graph, with values in $\R_{\geq 0}$; hence, a centrality measure assigns a {\em score} (a nonnegative
	real number) to each node, and a node is deemed ``important" if it receives a (relatively) high score. 
	
	Dozens of node  centrality measures have been proposed in the literature, often rooted in different notions of ``importance" and thus based on
	different mathematical principles, but most of them relying only on the underlying graph's internal structure. For 
	excellent introductions to the most widely used centrality measures, including references and examples, the reader is referred
	to the books by Estrada \cite{EstradaBOOK} and by Newman \cite{NewmanBOOK}.
	In some cases, one or more parameters are added to the purely graph-theoretic structure, possibly as a way to model some
	external influence on the network.  For example, one may want to take into account the existence of tensions between members of a social network \cite{EBH12}, or of heightened risk conditions for an economic or financial network  \cite{BBCGE}.
	
	A basic mathematical question is whether a given centrality measure is able to discriminate between (i.e., assign different scores to) nodes
	that are ``structurally different". Some centrality measures have been shown to have a higher ``discriminating power" than others, in the sense that the set of graphs for which they are able to assign different centrality values to at least two nodes in the graph
	contains strictly the corresponding set of graphs for the other measures \cite{EstradaArxiv,RP}. 
	In this paper we focus on one such measure, known as {\em subgraph centrality}, first proposed in  \cite{ERVsubgraphdefinition} (see
	also \cite{EHstatisticalmechanics}) and defined as follows: if $\G$ is a graph with $n$ vertices with adjacency matrix $A$, the
	subgraph centrality of the i$th$ node is given by
	\[ SC(i,\beta) = [e^{\beta A}]_{ii}\,,\]
	where $\beta > 0$ is a parameter. That is, the subgraph centrality of node $i$ is the $i$th diagonal entry of the matrix exponential 
	$e^{\beta A}$. 
	The interpretation of this centrality measure in terms of closed walks on the graph will be discussed in the next section. 
	It has been shown in  \cite{ERVsubgraphdefinition} that there exist regular graphs (i.e., graphs in which all nodes have the same degree) 
	which contain nodes with different subgraph centrality while the eigenvector centrality \cite{Bona}, Katz centrality \cite{Katz}, total communicability \cite{BKtotcomm}
	and many other measures of centrality are necessarily the same for all nodes. Besides its mathematical interest, the fact that some centrality 
	measures are able to provide a finer node ranking than others is clearly important also from the point of view of applications, since the whole
	point of these measures is to determine which nodes play a more central role within the network.
	
	Given two vertices $i,j$, we would like to compare their subgraph centralities $SC(i,\beta)$ and $SC(j,\beta)$. 
	When is one bigger than the other? When are they equal? Is it possible that  vertex $i$ is more important than $j$
	for some values of $\beta$, 
	but the reverse holds for some other values of $\beta$? If yes, how many times can two nodes reverse their relative position in the
	ranking as $\beta$ increases from 0 to infinity? We note that the latter question is not purely curiosity-driven, but it arises
	naturally in the context of risk-based centrality measures in economic and financial networks, see \cite{BBCGE}. In this paper we will consider
	these and other questions, answer some of them (in full or in part), and see which remain open.
	
	The remainder of the paper is organized as follows. In sections \ref{sec2.1:prel} and \ref{sec2.2:estrada}, after introducing the necessary notation and background
	information, we review recent results on subgraph centrality, walk regularity, and vertex cospectrality. In section \ref{sec2.3:interlace}
	we discuss the centrality interlacing problem, and in section \ref{sec2.4:resolvent} we consider another (related) centrality measure based on walks,
	resolvent centrality.  Finally, in section \ref{sec3:open} we provide a list of open questions.

	\section{Results on subgraph centrality}
	
	\subsection{Preliminaries}\label{sec2.1:prel}
	
	Let $\G$ be a simple undirected graph with $n$ vertices (we will use the words {\em node} and {\em vertex} interchangeably) and adjacency matrix $A$. Later we will generalize some of the results to the weighted and directed cases. The adjacency matrix is symmetric and thus diagonalizable: $A=QDQ^{-1}=QDQ^T$, with $D=\diag(\lambda_1, \ldots, \lambda_n)$ and
	$Q$ orthogonal. Hence, the $(i,i)$ entry of $A$ can be written as:
	\begin{equation*}
	A_{ii} = \sum\limits_{k=1}^n Q_{ik} \,\lambda_k \,[Q^{-1}]_{ki} =\sum\limits_{k=1}^n \lambda_k  (Q_{ik})^2.
	\end{equation*}
	Denote by $\mu_1, \ldots, \mu_d$ the distinct eigenvalues of $A$, without repetition. Then by grouping together equal eigenvalues, we obtain:
	\begin{equation*}
	A_{ii} = \sum\limits_{h=1}^d \mu_h \, C_{h i} \,, 
	\end{equation*}
	where $C_{h i} = \sum\limits_{k} Q_{ik} \,[Q^{-1}]_{ki} = \sum\limits_{k}  (Q_{ik})^2$, with the sum running over all indices $k$ such that $\lambda_k = \mu_h$.  This expression is useful, because it yields formulas for the individual entries of matrix functions, defined as
	\begin{equation*}
	f(A) = Q\, f(D) \,Q^{-1},
	\end{equation*}
	see \cite{Higham}.
	For example, the entry in position $(i,i)$ can be expressed as
	\begin{equation*}
	[f(A)]_{ii} =\sum\limits_{k=1}^n Q_{ik}\, f(\lambda_k) \,[Q^{-1}]_{ki} = \sum\limits_{h=1}^d f(\mu_h)\,  C_{h i}\,.
	\end{equation*}
	
	\begin{rem} 
		It is easy to verify that if $E_h$ denotes the 
		orthogonal projector onto the eigenspace of $A$ associated with $\mu_h$, then $C_{h i} = [E_h]_{ii}$; see, e.g., \cite{Meyer}.
	\end{rem}
	
	\vspace{0.3cm}
	
	The subgraph centrality of a vertex was introduced by Estrada and Rodr\'iguez-Vel\'azquez in \cite{ERVsubgraphdefinition} for $\beta=1$, and later generalized by Estrada and Hatano \cite{EHstatisticalmechanics} introducing the parameter $\beta$, which we assume always to be positive. 
	As mentioned in the Introduction, the subgraph centrality is defined as follows:
	\begin{equation*}
	SC(i,\beta) = [e^{\beta A}]_{ii}.
	\end{equation*}
	The subgraph centrality can be expressed in two ways. One way is to use the spectral decomposition of $A$, as described above; this is a compact and explicit expression, which we will use in some of our proofs.  Another way, which gives more insight into the combinatorial interpretation of this centrality measure, consists of using the Taylor series expansion at 0:
	\[[e^{\beta A}]_{ii} = 1 + \frac{\beta}{1!}  [A]_{ii}  + \frac{\beta^2}{2!} [A^2]_{ii}  + \frac{\beta^3}{3!} [A^3]_{ii} + \cdots .\]
	Recalling that for any positive integer $r$ the $(i,j)$ entry of $A^r$ equals the number of walks of length $r$  in $\G$ starting at node $i$ and ending at node $j$, we see
	that the subgraph centrality of node $i$ is a weighted sum of the number of closed walks (of any length) starting and ending in $i$, leading to the
	following  interpretation: vertices with a higher $SC$ score are visited by many closed walks (particularly shorter ones), and so are better connected to 
	other, not too distant  nodes in the graph. The role of the coefficients $\beta^r/r!$ in the Taylor expansion of $SC$ is that of penalizing longer walks. The
	parameter $\beta$ can be used to give more or less weight to longer or shorter walks. 
	
	\vspace{0.3cm}
	For our purposes, it is convenient to introduce the following terminology.
	\begin{defin} Two vertices $i,j$ of $\G$ are \textit{$\beta$-subgraph equivalent} if $[e^{\beta A}]_{ii} = [e^{\beta A}]_{jj}$.
	\end{defin}
	
	If there exists an automorphism of $\G$ that sends $i$ to $j$, then $i$ and $j$ are indistinguishable in the graph, so any centrality measure should give the same 
	score to both of them. In particular, it is obvious that two such $i$ and $j$ are $\beta$-subgraph equivalent. However, there are many examples of $\beta$-subgraph equivalent vertices not related by any automorphism.
	
	\vspace{0.3cm}
	
	We recall the definition of cospectral vertices from \cite{GScospectral}. By Theorem 3.1 of the same paper, there are many equivalent characterizations, of which we will use the following one.
	
	\begin{defin}Two vertices $i,j$ of $\G$ are \textit{cospectral} if for every integer $r \geq 0$, $[A^r]_{ii} = [A^r]_{jj}$.
	\end{defin}
	
	By Taylor series expansion, two cospectral vertices are $\beta$-subgraph equivalent. It can be easily shown that two cospectral vertices have also the same degree, eigenvector and resolvent centralities (defined below), while the same does not generally hold for other centrality measures, see \cite{RP,Stevanovic}.
	
	We recall also the definition of a walk regular graph, from \cite[p. 190]{GRbook} or \cite[Thm. 4.1]{GMKwalkregular}. 
	
	\begin{defin} $\G$ is \textit{walk regular} if for every integer $r \geq 0$, the diagonal entries of $A^r$ are constant; i.e.  $[A^r]_{ii} = [A^r]_{jj}$ for every $i,j$.
	\end{defin}
	
	Observe that a graph is walk regular if and only if every pair of vertices is cospectral.
	
	Another related quantity is the walk entropy of a graph \cite{EPHwalkentropy,EHstatisticalmechanics}, with parameter $\beta$.
	More precisely, the walk entropy of $G$ is defined as the Shannon entropy of the vector of subgraph centralities $\left( SC(1,\beta),  \ldots, SC(n,\beta)  \right)$, normalized to be a probability distribution, that is, 
	\begin{equation*}
	S(G,\beta) = \sum_{i=1}^n p_i \log p_i, \qquad p_i = \frac{[e^{\beta A}]_{ii}}{\text {Tr} [e^{\beta A}]}\,.
	\end{equation*}

	It can be easily seen that given $\beta > 0$ and $n$, the walk entropy is maximal for graphs whose $n$ vertices are all $\beta$-subgraph equivalent, and conversely
	(this follows from basic properties of Shannon entropy; see e.g. \cite{CTbook}).
	For the sake of simplicity, we reformulate the conjectures in section \ref{sec2.2:estrada} and Theorem \ref{teo: no interlacements with acc point} in our terminology, although in the original articles they were stated in terms of walk entropy or related concepts.

	\subsection{Walk regularity and proof of Estrada's Conjecture}\label{sec2.2:estrada}
	
	A natural question is whether two $\beta$-subgraph equivalent vertices are necessarily cospectral, and under which conditions. Estrada, and subsequently Kloster, Kr\'al and Sullivan, posed the following conjectures, reformulated with our terminology:

	\begin{conj*}[Estrada, \cite{EstradaArxiv}, Conjecture 3; see also \cite{EPHwalkentropy}]
		A graph $\G$ is walk regular if and only if all vertices are $1$-subgraph equivalent.
	\end{conj*}
	
	\begin{conj*}[Kloster, Kr\'al, Sullivan, \cite{KKS}, Conjecture 5]
		A graph $\G$ is walk regular if and only if there exists a rational $\beta > 0$ such that all vertices are $\beta$-subgraph equivalent. 
	\end{conj*}

	\noindent These conjectures were recently proven true in \cite{BDvertexdistinction}, with the following result.
	
	\begin{teo}[Ballini, Deniskin, \cite{BDvertexdistinction}, Theorem 3]
		Let $\G$ be a simple undirected graph with adjacency matrix $A$ and let $\beta\neq 0$ be an algebraic number. 
		If two vertices $i,j$ are $\beta$-subgraph equivalent, then they are cospectral.
		
		\label{teo: BD lindemann weierstrass}
	\end{teo}
	
	The idea of the proof is the following. By writing explicitly the $(i,i)$ entry of  $e^{\beta A} = Q \,e^{\beta D} \,Q^{-1}$, we obtain:
	\begin{equation*}
	SC(i,\beta) = [e^{\beta A}]_{ii} = \sum\limits_{h=1}^d  e^{\beta \mu_h} \, C_{ h i}\,.
	\end{equation*}
	If $SC(i,\beta) = SC(j,\beta)$, then by the  Lindemann-Weiestrass Theorem (see, e.g., \cite{Baker}) 
	the coefficients must be equal: $C_{h i} = C_{h j}$ for every $h$. This implies that $i$ and $j$ are cospectral, by writing explicitly $[A^r]_{ii} = \sum\limits_h \mu_h^r\, C_{h i } $. 
	
	As stated in \cite{BDvertexdistinction}, the result also holds under weaker assumptions. For example, it is applicable
	to directed graphs with diagonalizable adjacency matrix $A$, and 
	to weighted undirected graphs with algebraic weights.  
	
	\vspace{0.3cm}
	Before this result, it was known (see \cite{HKS,KKS}) that there exist graphs for which there are
	values of $\beta$ for which all vertices are $\beta$-subgraph equivalent, but that are not walk regular. 
	As a consequence of the result in \cite{BDvertexdistinction}, we know that such values of $\beta$ are
	necessarily transcendental.
	
	\vspace{0.3cm}
	
	There are, however, more open questions waiting to be answered. 
	The following result is a slight modification of Theorem 2.2 in  \cite{Benzi}. With our terminology, the condition of walk regularity is equivalent to all vertices being cospectral, and the maximum walk entropy property (see \cite{Benzi}) is equivalent to all vertices being $\beta$-subgraph equivalent. By comparing only two vertices and not all of them together, we can reformulate the theorem as follows:

	\begin{teo}
		Given two vertices $i,j$ of $\G$, let $I$ be the set of all $\beta>0$ such that $i$ and $j$ are $\beta$-subgraph equivalent. 
		If there is a $\beta_0\in I$ which is an accumulation point for a sequence in $I$, then $i$ and $j$ are cospectral and $I$ is necessarily all of $\R^+$.
		\label{teo: no interlacements with acc point}
	\end{teo}
	\begin{proof}
		Consider the function $g(\beta) = [e^{\beta A}]_{ii} - [e^{\beta A}]_{jj}$, which is an analytic function of $\beta$. The zeros of a non-constant analytic function (if there are any) form a discrete set (possibly of infinite cardinality), so the hypothesis that $\beta_0$ is an accumulation point implies that $g$ is identically zero on all $\R^+$. This means that the coefficients of the Taylor series expansion are all zero. The $r$-th coefficient is $\frac{1}{r!} \left([A^r]_{ii} - [A^r]_{jj} \right)$, so if it is equal to zero for all $r$, then $i$ and $j$ must be cospectral. For more details, see the proof of Theorem 2.2 in \cite{Benzi}.
	\end{proof}
	
	It is natural to ask if this result can be improved. For example, is it still true if $I$ is finite? Is there any condition on $M=|I|$, the cardinality of $I$, under which the result holds?
	
	It has been shown (see, e.g., \cite{KKS}) that the result does not hold  for $M=1$, and  it is common to find pairs of non-cospectral vertices which are $\beta$-subgraph equivalent at least for one $\beta$. In \cite{BKcentralitymeasures}, it was shown that if $deg(i) > deg(j)$, where $deg(\cdot)$ is the degree of a vertex, then for all $\beta$ small enough $SC(i,\beta) > SC(j,\beta)$. Similarly, with $EC(\cdot )$ being the eigenvector centrality,\footnote{We recall that for a connected graph $\G$, the eigenvector centrality of vertex $i$ is defined as the $i$th entry of the unique strictly positive, normalized
		eigenvector associated with the largest eigenvalue of $A$.}
	if $EC(i) > EC(j)$ then $SC(i,\beta) > SC(j,\beta)$ for all $\beta$ large enough.
	Thus if for two vertices $i,j$ we have $deg(i)> deg(j)$ and $EC(i) < EC(j)$, then by continuity there exists (at least) one value of $\beta$ for which $SC(i,\beta) = SC(j,\beta)$.

	\begin{defin}
		Let $i,j$ be two non-cospectral vertices of a graph $\G$. We say that $i$ and $j$ \textit{interlace at $\beta$} if $SC(i,\beta) = SC(j,\beta)$; $\beta$ is an 
		\textit{interlacing value}.
	\end{defin}
	
	\begin{rem} \textnormal{By virtue of Theorem \ref{teo: BD lindemann weierstrass}, it follows that given two vertices $i,j$, any interlacing value $\beta$ cannot be an algebraic number, so it must necessarily be transcendental.}
	\end{rem}
	
	\begin{rem}\textnormal{ Let $I$ be the set of all interlacing values of two vertices $i,j$. Then $I$ is the zero set of the analytic function (restricted to $\beta>0$):
			\[ g(\beta) = [e^{\beta A}]_{ii} - [e^{\beta A}]_{jj}.\]
			Theorem \ref{teo: no interlacements with acc point} implies that $I$ must be either a discrete set (possibly infinite)} or all of $\R^+$.
	\end{rem}
	
	The fact that the relative ranking of two nodes $j$ and $j$ can be different for different values of $\beta$ is of
	interest in applications to Network Science. For instance, in the study of financial markets 
	the parameter $\beta$ can be a measure of the risk affecting a given market, and as it varies some assets will swap
	places in a ranking based on subgraph centrality or other parameter-dependent measures, see \cite{BBCGE}. One
	measure of the sensitivity of the ranking is how many times the relative ranking of two assets can change as $\beta$ varies.
	It is therefore of interest to provide an upper bound on the number of interlacing values. This question will be addressed next.

	\subsection{Bounding the number of interlacing values}\label{sec2.3:interlace}
	
	We have seen that if nodes $i$ and $j$ are not cospectral, then the set $I$ cannot have an accumulation point. 
	The following result from \cite{BBCGE} states that $I$ cannot be infinite, under a mild condition:
	
	\begin{teo}[Bartesaghi \textit{et al.},  \cite{BBCGE}, Theorem 7.2]
		Let the undirected graph $\G$ be connected, and let 
		$i$ and $j$ be two nodes with different eigenvector centrality. Then the number of interlacing values for $i$ and $j$ is necessarily finite (possibly zero).
		\label{teo: no infinite interlac}
	\end{teo}
	
	We will now show that $I$ is always finite (i.e., the condition on the eigenvector centrality can be dropped), and that it is possible to find an upper bound for the cardinality of $I$. Indeed, we will show that it is always less than the number of vertices $n$, even for certain types of weighted and directed graphs; for simple undirected graphs, sharper estimates can be given. These are hiterto unpublished results, extracted from the first author's  thesis \cite{Deniskin}.
	We will first need the following technical result.
	
	\begin{prop}
		Let $\lambda_1, \ldots, \lambda_n$ be distinct real numbers, and $c_1, \ldots , c_n $ (not necessarily distinct) real numbers. Consider the following function:
		\begin{equation*}
		f(\beta)=c_1 e^{\beta \lambda_1} + c_2 e^{\beta \lambda_2} + \cdots + c_n e^{\beta \lambda_n} .
		\end{equation*}
		
		Assume that $f(\beta)$ is not constant. Then the equation  $f(\beta)=0$ has at most $n-1$ solutions, and the equation $f(\beta)=a$ with $a \neq 0$ has at most $n$ solutions.
		
		\label{prop: max d solutions to sum exp = 0}
	\end{prop}
	\begin{proof}
		The function $f(\beta)$ is constant only if all $c_i$ are equal to zero, or if there is only one $j$ such that $c_j \neq 0$ and $\lambda_j=0$.
		
		We proceed by induction on $n$. For $n=1$ we have $f(\beta)=c_1 e^{\beta \lambda_1}$, which is an injective function and never equal to zero if  $c_1 \neq 0$ 
		and $\lambda_1 \neq 0$.
		
		We now assume that the claim holds for $n$ and for every function $f_n(\beta)$ with $n$ terms, and we will prove it for $f_{n+1}(\beta)=c_1 e^{\beta \lambda_1} +  \cdots +c_n e^{\beta \lambda_n} + c_{n+1} e^{\beta \lambda_{n+1}}$.
		
		Without loss of generality assume that $c_{n+1}\neq 0$. Then:
		\begin{equation*}
		f_{n+1}(\beta)=0 \hspace{3mm} \iff \hspace{3mm} \frac{f_{n+1}(\beta)}{c_{n+1} e ^{\beta \lambda_{n+1}}} =0,
		\end{equation*} 
		\begin{equation*}
		\begin{aligned}
		\frac{f_{n+1}(\beta)}{c_{n+1} e ^{\beta \lambda_{n+1}}} 
		=& \frac{c_1}{c_{n+1}} e^{\beta(\lambda_1-\lambda_{n+1})} + \cdots + \frac{c_n}{c_{n+1}} e^{\beta(\lambda_n-\lambda_{n+1})} +1=\\
		=& d_1 e^{\beta \mu_1} + \cdots + d_n e^{\beta \mu_n} + 1=g_{n}(\beta)+1,
		\end{aligned}
		\end{equation*}
		with $d_i=\dfrac{c_i}{c_{n+1}}$ and $\mu_i=\lambda_i-\lambda_{n+1}$.  We have $ \frac{f_{n+1}(\beta)}{c_{n+1} e ^{\beta \lambda_{n+1}}} =0$ if and only if $g_n(\beta)=-1$; however, $g_n(\beta)$ is the sum of $n$ exponential terms, so by inductive hypothesis there are at most $n$ solutions to $g_n(\beta)=-1$, which are also all the solutions to $f_{n+1}(\beta)=0$.
		
		For $a \neq 0$, suppose that the equation $f_{n+1}(\beta)=a$ has $k$ solutions $\beta_1 < \beta_2 <\ldots < \beta_k$. By Rolle's Theorem, there exists $k-1$ numbers $\gamma_1, \ldots, \gamma_{k-1}$ such that $\beta_i < \gamma_i < \beta_{i+1}$ and $f_{n+1}'(\gamma_i)=0$ for $1 \leq i \leq n-1$. Now,
		\begin{equation*}
		f_{n+1}'(\gamma) = c_1\lambda_1 e^{\gamma \lambda_1} + \cdots + c_n \lambda_n e^{\gamma \lambda_n} + c_{n+1} \lambda_{n+1} e^{\gamma \lambda_{n+1}} .
		\end{equation*} 	
		If there exists at least one $i$ such that $c_i \lambda_i\neq 0$, then by  what we have just proven the equation $f_{n+1}'(\gamma)=0$ has at most $n-1$ solutions. This means $k-1 \leq n-1$ and so $k \leq n$.
		
		If $c_i \lambda_i=0$ for every $i$, then the function $f_{n+1}'(x)$ is identically zero, so $f_{n+1}(x)$ is constant, which contradicts the initial hypothesis. 
	\end{proof}
	
	Using this result we can derive the following bound on the number of interlacing values. 
	
	\begin{teo}\label{inter_points}
		Let $\G$ be a (possibly weighted) undirected graph with adjacency matrix $A$, and denote by $d$ the number of distinct eigenvalues of  $A$.  
		For any two non-cospectral vertices $i,j$ there can be at most $d-1$ interlacing values for the subgraph centrality.	
		\label{teo: SC max d interlacements}
	\end{teo}
	\begin{proof}
		Write 
		\[g(\beta) = [e^{\beta A }]_{ii} - [e^{\beta A}]_{jj} = \sum\limits_{h=1}^d e^{\beta \mu_h}\,  \left( C_{h i} - C_{h j} \right) \,.\]
		By applying Proposition \ref{prop: max d solutions to sum exp = 0}, the equation $g(\beta) = 0$ has at most $d-1$ solutions, if $g$ is not identically constant. This means that there are at most $d-1$ interlacing values of $\beta$, unless the vertices are cospectral.	
	\end{proof}
	
	\begin{rem} \textnormal{
			We note that the result also holds  for the case of  a digraph with real
			spectrum and diagonalizable adjacency matrix.}
	\end{rem}
	
	%%%%%%%%%%%%%%%%%%%%
	%%%%%%%%%%%%%%%%%%%%
	Using the theory of totally positive matrices, we can prove a slightly stronger result.
	
	\begin{defin}
		A matrix $B$ of size $m\times n$ is called \textit{totally positive} if every square submatrix of $B$ has positive determinant.
	\end{defin}

	Below are two useful results on totally positive matrices, whose proofs can be found in  \cite{PinkusBOOK}.
	\begin{teo}[Pinkus,  \cite{PinkusBOOK}, Theorem 3.3]
		Let $B$ be a totally positive $m\times n$ matrix. Let $x\in \R^n$ be a column vector, and $y = Bx$. 
		Then the number of sign changes  in $y$ is less than  or equal to the number of sign changes in $x$. 
		\label{teo: TP decreases sign changes}
	\end{teo}
	
	For our purposes, zeros are simply ignored when counting the number of sign changes in a vector.

	\begin{teo}[Pinkus,  \cite{PinkusBOOK}, Section 4.2]
		Let $\beta_1 < \ldots < \beta_M$ and $\mu_1 < \ldots < \mu_d$ be real numbers. The $M\times d$ matrix $B = \left [e^{\beta_l \mu_h}\right ]$ is totally positive.
		\label{teo: matrix exp is TP}
	\end{teo}

	\noindent We can now state and prove the following result.
	
	\begin{prop}
		Let $\G$ be a (possibly weighted) undirected  graph with adjacency matrix $A$. 	For any two non-cospectral vertices $i,j$, let $C_{h i}$ and $C_{h j}$ be defined as above. Consider the function 
		\[g(\beta) = [e^{\beta A}]_{ii} - [e^{\beta A}]_{jj} =\sum\limits_{h=1}^d  e^{\beta \mu_h} \, \left( C_{h i} - C_{h j} \right) \,.\]	
		The number of interlacing values for $i$ and $j$ (which is the number of zeros of $g(\beta)$) is less than or equal to the number of sign changes in the vector $C_{h i} - C_{h j}$, as $h$ varies from $1$ to $d$. 
	\end{prop}
	\begin{proof}
		Let $\beta_1, \ldots, \beta_M$ be any $M$ distinct positive real numbers, numbered in increasing order.
		By Theorem \ref{teo: matrix exp is TP}, the matrix $B$ with entries $B_{l h} = e^{\beta_l \mu_h}$ is totally positive. Let $x\in \R^d$ be a vector with components $x_h =  C_{h i} - C_{h j}$.  The product $y = Bx$ has components $y_l = g(\beta_l)$. Denote by $s$ the number of sign changes in $y$. By Theorem \ref{teo: TP decreases sign changes}, $s$ is less than or equal to the	number of sign changes in $x$, which are at most $d-1$. Since the $\beta_l$'s and $M$ are arbitrary, and the zeros of $g(\beta)$ form a discrete set, it follows that the function $g(\beta)$ has at most $s$ zeros. Therefore there can be at most $s$ interlacing values for $i$ and $j$. 
	\end{proof}
	\begin{rem} 
		Since the number of sign changes is obviously less than  $d-1$,  we recover Theorem \ref{inter_points}. However, if we know that $C_{h i} > C_{h j}$ for some values of $h$, then we can obtain a better bound for the number of interlacing values.
	\end{rem}
	
	In practice, we found that  the bound $d-1$ dramatically overestimates the typical number of interlacing values. 
	Computational experiments show that for many simple undirected graphs with between 10 and 50 vertices, the number of interlacing values
	for the subgraph centrality is at most 2 (see Fig.1 for the smallest example we could find of a graph with two such values).
	The adjacency matrices of such graphs are 0-1 matrices, and this imposes a number
	of constraints on the eigenvalues and eigenvectors. Hence, such matrices are highly non-generic, and thus it is not surprising that 
	the generic bound $d-1$ is a wild overestimate on the number of interlacing points.  Obtaining sharper bounds for specific classes of
	graphs represents an open problem. We return on this topic in section \ref{sec3:open}.
	
	In the computational experiments we performed, for each graph a suitable interval has been chosen, and the subgraph centrality was calculated on a discretization of that interval. Interlacing values are found as zeros of the function $g(\beta)$, which can be deduced from sign changes in the function calculated on the discretization mesh. 
	
	This, of course, does not constitute a formal proof that the graphs in Fig.1 and Fig.2 are the minimal ones for which interlacing occurs. They are the smallest ones we 
	have managed to find through a brute-force search over all small graphs ($n \leq 10$). Proving the minimality of these examples remains an open question.
	
	On the other hand, cospectrality is a condition which can be easily verified or disproved computationally. Indeed, by the Cayley-Hamilton theorem, it is sufficient to calculate $A^r$ for $0 \leq r \leq n-1$, which can be done in exact arithmetic since $A$ has $0$-$1$ entries.

	\begin{figure}[!t] \label{FIG1}
		\begin{subfigure}{0.39\textwidth}
			\begin{tikzpicture}[x=22, y=22]
			\tikzset{palla/.style={circle, draw=black, minimum size = 0.5cm,outer sep =0, inner sep=0}}
			\tikzmath{\l=2.5;};
			
			\node[palla] (1) at (0,0) {1};
			\node[palla] (8) at (\l,0) {8}; 
			\node[palla] (7) at (\l,\l) {7};
			\node[palla] (5) at (0,\l) {5}; 
			\node[palla] (6) at (0.5*\l,1.4*\l) {6};
			
			\node[palla] (2) at (190:2.1) {2}; 
			\node[palla] (3) at (260:2.1) {3};
			\node[palla] (4) at (225:2.6) {4}; 
			
			\node[palla] (9) at (1.4*\l,-0.7*\l) {9};

			\draw (4) -- (1) -- (2) -- (4) -- (3) -- (1) -- (5) -- (7) -- (8) -- (1);
			\draw (5) -- (6) -- (7); \draw (8) -- (9);
			
			\end{tikzpicture}
		\end{subfigure}
		\begin{subfigure}{0.59\textwidth}
			\includegraphics[width=\textwidth]{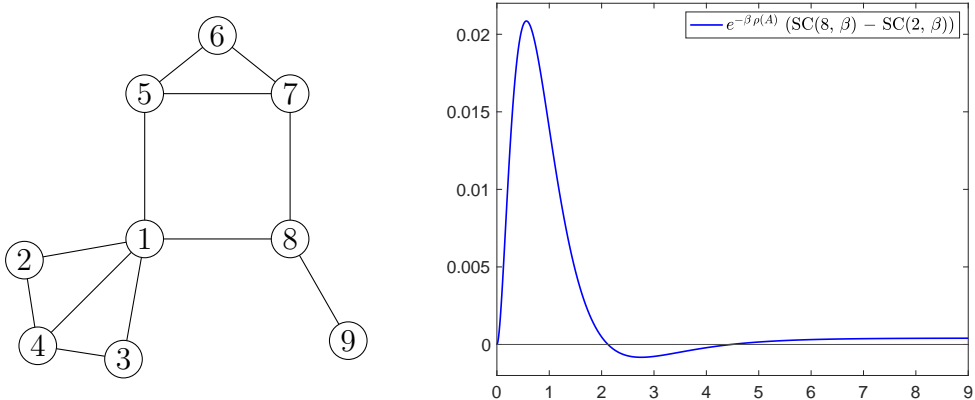}
			
		\end{subfigure}

		\caption{Graph with $n=9$ nodes. It is the smallest graph we have found with at least two interlacing points: for vertices 2 and 8, the interlacing values are approximately $\beta \approx 2.12,4.48$. }
	\end{figure}

	\subsection{Resolvent Centrality}\label{sec2.4:resolvent}
	Let  $\rho$ be the spectral radius of $A$ and let $\alpha \in (0, \frac{1}{\rho})$. The \textit{resolvent centrality} \cite{EH10} (also called \textit{resolvent subgraph centrality}) with parameter $\alpha$ is defined as follows:
	\[RC(i,\alpha) = [(I-\alpha A )^{-1}]_{ii}.\]
	This centrality measure is related to, but distinct from, Katz centrality (see, e.g.,  \cite{NewmanBOOK}). 
	
	The idea, analogously to  (exponential) subgraph centrality, is to use as centrality scores the diagonal entries $f(A)_{ii}$, where $f(A)$ is a matrix function. The resolvent centrality corresponds to $f(x) = \frac{1}{1-\alpha x}$ in place of $f(x) = e^{\beta x}$. The inverse of a matrix can be expressed using Cramer's rule: $[B^{-1}]_{ij} = (-1)^{i+j}\,\frac{\det(B_{[j,i]} )}{\det B}$, where  $B_{[j,i]}$ denotes the matrix obtained removing row $j$ and column $i$ from $B$. Therefore:
	\begin{equation*}
	RC(i,\alpha) = [(I-\alpha A )^{-1}]_{ii} = \frac{\det(I -\alpha A_{[i,i]} ) }{ \det(I -\alpha A)}.
	\label{eq: RC rapporto Cramer polinomio}
	\end{equation*}
	Note that for $0 < \alpha < \frac{1}{\rho}$ the resolvent centrality is well defined (and positive), both by Taylor series expansion $(I-\alpha A)^{-1} = I + \alpha A + \alpha^2 A^2 + \cdots$, which converges because $\rho(\alpha A ) < 1$, and by Cramer's rule since for $\alpha < \frac{1}{\rho}$, the denominator $\det(I -\alpha A)$ is not zero.
	
	We present two upper bounds  for the number of  interlacing values of resolvent centrality. The first one is valid for any matrix $A$, while the second is sharper but restricted only to simple undirected graphs.
	
	\begin{teo}
		Let $\G$ be a (possibly weighted) directed or undirected graph with $n$ vertices and with adjacency matrix $A$. For any two non-cospectral vertices $i,j$, there can be at most $n-1$ interlacing values for resolvent centrality, i.e. at most $n-1$ solutions to $RC(i,\alpha)=RC(j,\alpha)$.
	\end{teo}
	\begin{proof}
		Suppose that there are $n$ or more interlacing values, at points $\alpha_1, \ldots, \alpha_n$. Together with equation (\ref{eq: RC rapporto Cramer polinomio}), this implies that
		\[\det(I -\alpha_k A_{[i,i]}) = \det(I -\alpha_k A_{[j,j]}) \; \; \; \text{ for } 1 \leq k \leq n. \]
		
		These determinants are polynomials of degree $n-1$ in $\alpha$, so if they coincide on $n$ values of $\alpha$, they must be the same polynomial. This implies that $RC(i,\alpha) = RC(j,\alpha)$ for every $\alpha$.
		
		As in the case with subgraph centrality, $RC(i,\alpha)$ and $RC(j,\alpha)$ are analytic functions, with $r$th coefficient $[A^r]_{ii}$ and $[A^r]_{jj}$ respectively. If the two functions coincide, they must have the same coefficients. This is equivalent to $i$ and $j$ being cospectral.
	\end{proof}
	
	\begin{teo}
		Let $\G$ be a simple undirected graph with adjacency matrix $A$, and denote by $d$ the number of distinct eigenvalues of  $A$.  
		For any two non-cospectral vertices $i,j$ there can be at most $d-1$ interlacing values for the resolvent  centrality. % i.e. at most $d-1$ solutions to $RC(i,\alpha)=RC(j,\alpha)$.
	\end{teo}
	\begin{proof}
		As above, we have $RC(i,\alpha) = RC(j,\alpha)$ if and only if $\det(I -\alpha A_{[i,i]}) = \det(I -\alpha A_{[j,j]} )$. 
		Cauchy's Interlacing Theorem (see \cite[Thm.~4.3.17]{HJMA}) states  that if $\lambda_1 \leq \lambda_2 \leq \ldots \leq \lambda_n$ are the eigenvalues of 
		the $n\times n$ symmetric matrix $A$, and $\nu_1 \leq \ldots \leq \nu_{n-1}$ are the eigenvalues of the submatrix $A_{[i,i]}$, then
		\[\lambda_1 \leq \nu_1 \leq \lambda_2 \leq \nu_2 \leq \ldots  \le \nu_{n-1} \leq \lambda_n. \]
		
		If an eigenvalue has multiplicity $t$, e.g. $\lambda_k = \lambda_{k+1} = \ldots  = \lambda_{k+t-1}$, then $\nu_k = \ldots = \nu_{k+t-2}$, so $\nu_k$ has multiplicity 
		at least $t-1$. 
		
		Denote by $\mu_1, \ldots , \mu_d$ the eigenvalues of $A$ without repetition, and let $m_h$ be the multiplicity of $\mu_h$. Then by Cauchy's Interlacing Theorem, $\mu_h$ is an eigenvalue of $A_{[i,i]}$ with multiplicity at least $m_h-1$. So $A_{[j,j]}$ and $A_{[i,i]}$ have at least $\sum_h m_h-1  = n-d$ common eigenvalues, and their characteristic polynomials have a common factor $q(\alpha)$ of degree $n-d$. 
		
		Suppose that $i,j$ have $d$ interlacing values $\alpha_1, \ldots, \alpha_d$ in $(0, \frac{1}{\rho})$. As noted above, the resolvent subgraph centrality is strictly positive for these values of $\alpha$, so the polynomials $\frac{\det(I -\alpha A_{[i,i]})}{q(\alpha)} $ and $\frac{\det(I -\alpha A_{[j,j]})}{q(\alpha)}$ have degree $d-1$ and coincide on $d$ values. This implies that $\det(I -\alpha A_{[i,i]})$ and $\det(I -\alpha A_{[j,j]})$ are the same polynomial, and therefore $i$ and $j$ are necessarily cospectral.
		
	\end{proof}
	
	\begin{figure}[!t]
		\centering
		\begin{subfigure}{0.99\textwidth}
			\centering
			\begin{tikzpicture}[x=0.76cm, y=0.76cm]
			\tikzset{palla/.style={circle, draw=black, minimum size = 0.7cm,outer sep =0, inner sep=0}}
			\tikzmath{\l=2;};
			
			\node[palla] (1) at (-0.5,0.5) {1};
			\node[palla] (2) at (0.5,-0.5) {2}; 
			\node[palla] (3) at (-3,-3) {3};
			\node[palla] (4) at (3,-3) {4}; 
			\node[palla] (5) at (1.4,1.4) {5};
			\node[palla] (6) at (3,3) {6}; 
			\node[palla] (7) at (-3,0) {7};
			\node[palla] (8) at (0,-3) {8}; 
			\node[palla] (9) at (3,0) {9};
			\node[palla] (10) at (0,3) {10}; 
			
			\draw (1) -- (7) -- (2) -- (8) -- (1) -- (9) -- (2) -- (10) -- (1);
			\draw (9) -- (4) -- (8) -- (3) -- (7) -- (10) -- (5) -- (9) -- (6) -- (10);
			
			\end{tikzpicture}	
		\end{subfigure}
		
		\begin{subfigure}{0.49\textwidth}
			\includegraphics[width=1\textwidth]{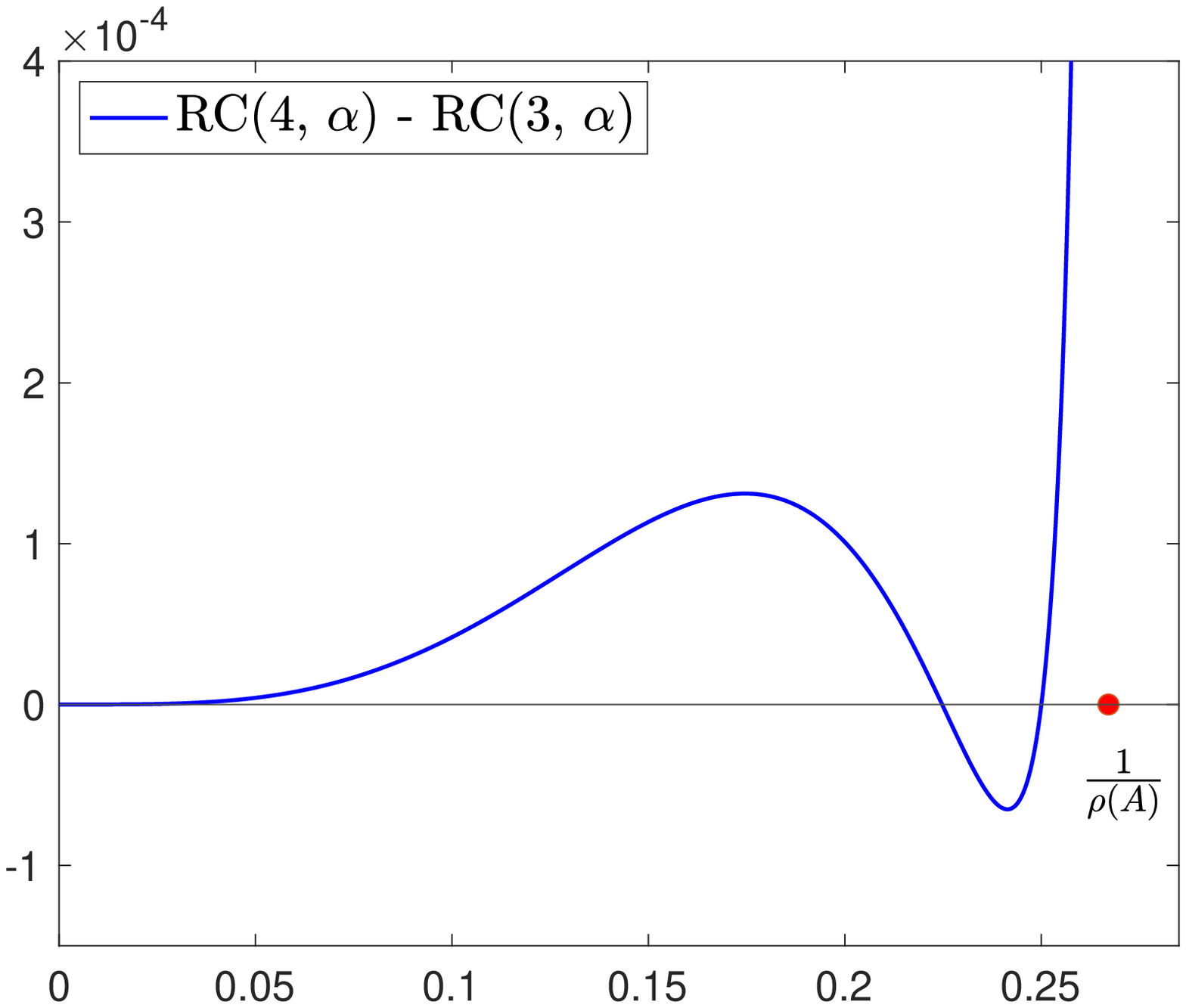}
		\end{subfigure}
		\begin{subfigure}{0.49\textwidth}
			\includegraphics[width=1\textwidth]{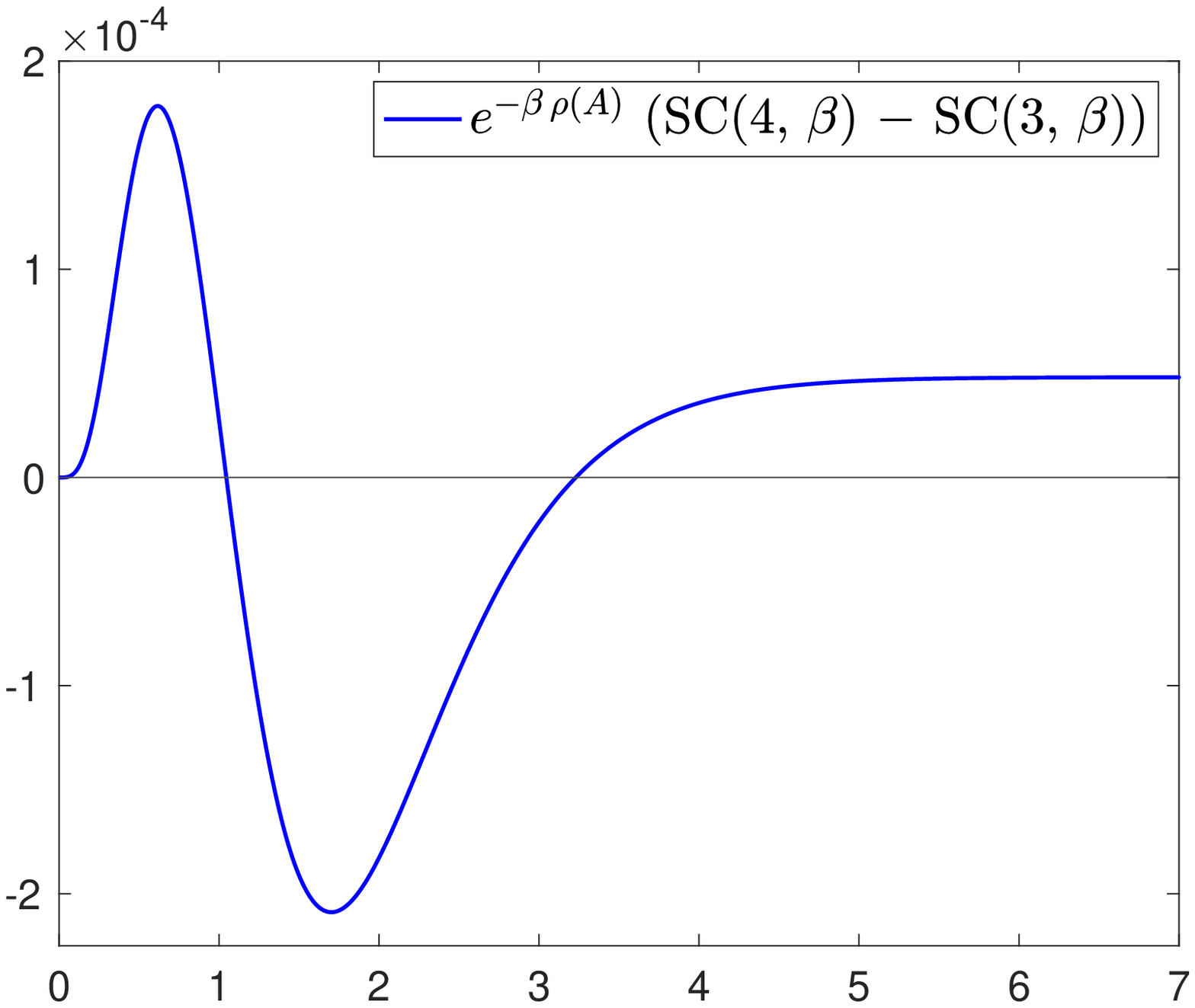}
		\end{subfigure}
		
		\caption{Graph with $n=10$ nodes. It is the smallest graph we have found with at least two interlacing values for resolvent  centrality. Vertices 3 and 4 interlace. On the left 
			we plot the difference $RC(4, \alpha) - RC(3,\alpha)$. On the right we plot the rescaled difference $e^{-\beta \rho} \,(SC(4, \beta) - SC(3,\beta) )$, showing that in
			this example the subgraph centralities of the two nodes also interlace twice.}
	\end{figure}

	In Fig.2 we show an example of interlacing for resolvent centrality and the corresponding behavior of the subgraph centrality.
	
	\vspace{0.3cm}
	
	In this paper we have focused on exponential- and resolvent-based subgraph centrality. The use of the matrix exponential and resolvent to define
	centrality measures is well established, both because of their demonstrated effectiveness in practice and because of the fundamental role that these two matrix functions
	play in spectral theory and more generally in mathematical physics. We refer to \cite{EBH12} for a thorough discussion of how physical
	analogies  naturally lead to the use of these two matrix functions in applications to Network Science.
	
	The reader may wonder what happens if another matrix function $f(A)$ is chosen, instead of the exponential or the resolvent. This is known as \textit{functional centrality} \cite{REGfuncentr}, and can be defined under reasonable assumptions on $f$. Typically, one asks for $f$ to be real analytic in a neighbourhood of $0$ (to work with nice and regular enough functions) and for all its MacLaurin series' coefficients to be positive (to preserve the combinatorial interpretation in terms of walks on the graph), see
	\cite{BKcentralitymeasures}. Such a function is called \textit{admissible}, and by denoting $R$ its radius of convergence in $0$, the matrix function $f(\beta A)$ is defined for any $\beta \in (0, \frac{R}{\rho(A)})$, where $\rho(A)$ is the spectral radius of $A$. If $f$ is entire, then $\beta$ can take any value in $(0,\infty)$. 
	In \cite{BKcentralitymeasures} the limiting behaviour of functional centrality was analyzed, for the two cases $\beta \to 0^+$ and $\beta \to \left(\frac{R}{\rho(A)}\right)^-$
	(or $\beta \to \infty$ in the entire case), showing that in these limits the behaviour of these centrality measures is identical to that of (exponential) subgraph and resolvent centrality.
	Similarly to these centralities, interlacing may occur also for functional centrality. 
	It is possible to adapt the proof of Theorem \ref{teo: no infinite interlac} also for this case, concluding that under the same assumptions the number of interlacing points cannot be infinite. However, without more information on the function $f$, it is difficult to prove a quantitative bound on the number of interlacing values as we have done with subgraph and resolvent centralities.

	\section{Open questions}\label{sec3:open}
	
	In this last section we propose some open problems which would be of interest to investigate further.
	
	Theorem \ref{teo: BD lindemann weierstrass}, as also noted in \cite{BDvertexdistinction}, holds for any diagonalizable matrix $A$ whose entries are all algebraic numbers (and thus for weighted undirected graphs with algebraic weights). Indeed, for undirected graphs the adjacency matrix is symmetric, and therefore diagonalizable. The algebraicity condition is necessary: we need it to apply  the Lindemann-Weierstrass Theorem, and there exists interlacing values for transcendental $\beta$. However, the diagonalizability condition is not strictly necessary and it was shown 
	in \cite{BDvertexdistinction} that  it can be somewhat relaxed by imposing some constraints on the Jordan structure of $A$. At present we do not know if the 
	diagonalizablity condition can be completely removed or not (without restrictions on the Jordan form). 
	This leads to formulating the following problem.
	
	\begin{prob}
		Let $A$ be a (non diagonalizable)  $n\times n$ matrix, whose entries are all nonnegative algebraic numbers, regarded as the adjacency matrix of a weighted directed graph $\G$. Let $\beta\neq 0$ be an algebraic number. Prove or disprove the following statement: if for two vertices $i$, $j$ we have $[e^{\beta A}]_{ii} = [e^{\beta A}]_{jj}$, then $i$ and $j$ are cospectral.
	\end{prob}
	
	Next, we turn to the problem of estimating the number of interlacing values.
	We checked many graphs computationally, and the number of interlacing values was always much smaller (other than for the smallest graphs) than the bounds we found. 
	For the majority of pairs of vertices the subgraph centralities do not interlace at all, or at most once (it is common to find vertices for which degree and eigenvector centrality give different rankings). The smallest graph we could find (through a computer search) with two interlacing values  of exponential subgraph centrality has 9 vertices, shown in Fig.1. When $n$ is of order $20$ or more, it's not too uncommon to find pairs of vertices with two interlacing values; however, we found only a couple of graphs with 3 interlacing values, and none with 4 or more
	(see Fig.3). Hence, we propose the following problem.
	
	\begin{prob} 
		Given $k \in \N$, characterize the smallest graph(s) with a pair of vertices that have at least $k$ interlacing values of the subgraph centrality.
		Conversely, given a graph with $n$ vertices, what is the largest number of possible interlacing values for a pair of its vertices? In particular, are there
		classes of graphs for which the bound $d-1$ is attained?
	\end{prob}

	\begin{figure}[!t]
		\centering
		\includegraphics[width=0.99\textwidth]{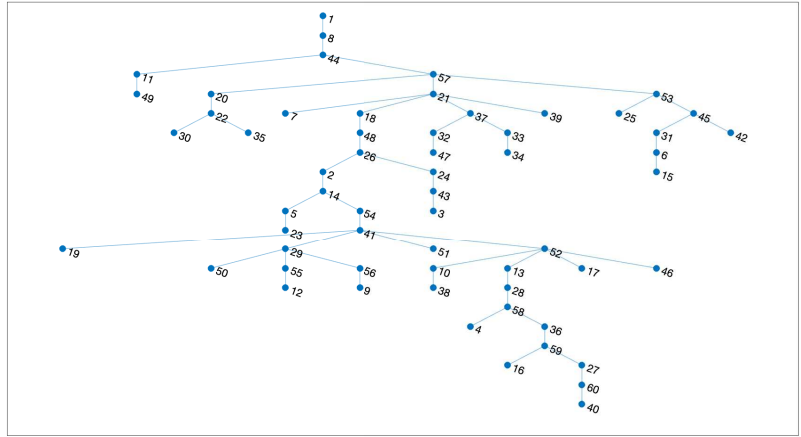}
		\includegraphics[width=0.99\textwidth]{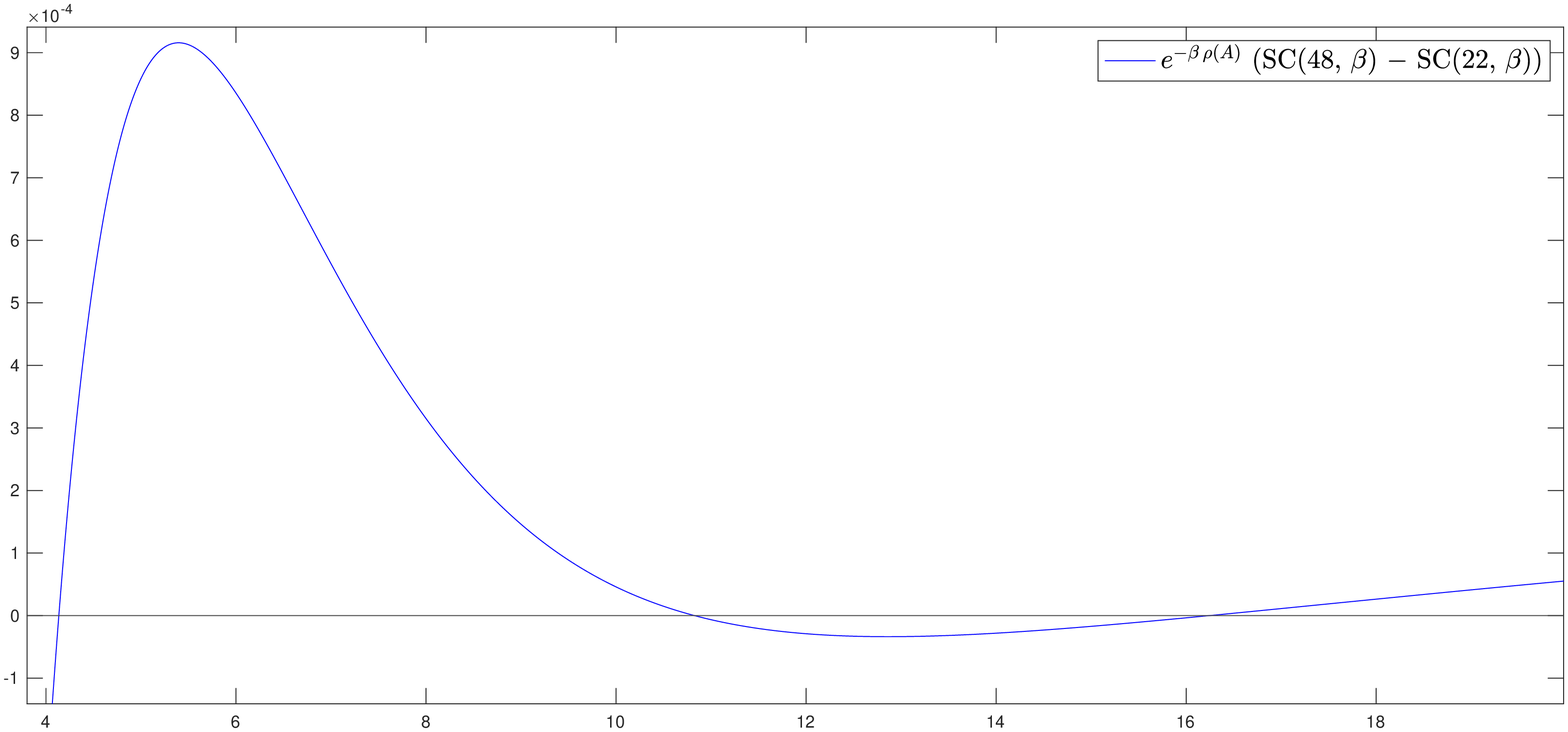}
		\caption{Graph with $n=60$ vertices %(this graph turns out to be a tree)
			and three interlacing values. Vertices 22 and 48 interlace three times, with interlacing values $\beta \approx 4.13, 10.9, 16.3$. Unlike the graphs in figures 1 and 2, this example was not found by complete search of all graphs up to 60 vertices, so it is quite possible that there are smaller graphs with three interlacing values.} 
	\end{figure}

	A related, more informal problem is the following: is there a simple way to construct graphs with many interlacing values?
	\vspace{0.3cm}
	
	Theorem \ref{teo: SC max d interlacements} can be extended to a broader category of graphs. Indeed, it requires Proposition \ref{prop: max d solutions to sum exp = 0}, so it is sufficient to have diagonalizable $A$ with real eigenvalues. This requirement is necessary, because if $\lambda_i$ could be any complex numbers, the equation could have infinitely many solutions. 
	To see this, condider the following example:  
	\[ A= \begin{pmatrix}
	0 & -1 &0\\
	1 & 0 &0 \\
	0 & 0 & 0
	\end{pmatrix}, \quad
	e^{\beta A} = \begin{pmatrix}
	\cos(\beta) & -\sin(\beta)  & 0\\
	\sin(\beta) & \cos(\beta)   & 0\\
	0 & 0 & 1
	\end{pmatrix}.\]
	Note that $A$ has eigenvalues $0, i, -i$. 
	Since $g(\beta) = [e^{\beta A}]_{11} - [e^{\beta A}]_{33} = \cos(\beta) -1$ has infinitely many solutions,  there are infinitely many interlacing values for the
	$\beta$-sugraph centrality of vertices 1 and 3. Hence, in general the statement is false for signed graphs. However, it could still be true for some other category of graphs. 
	In particular, we ask the following question.
	
	\begin{prob} 
		Does Theorem \ref{teo: SC max d interlacements} hold in the case of directed graphs?   
		(Here we allow the graph to be weighted with positive weights.)
	\end{prob}

	While we have formulated these questions for the subgraph centrality,  they are  also of interest for other centrality measures. We have seen some results for resolvent centrality, but the bound on the number of interlacing $\alpha$ values obtained is not optimal. Also, it would be interesting to extend the result to other centrality measures which are dependent on a parameter. For example, Katz centrality \cite{EstradaBOOK,NewmanBOOK} and total communicability \cite{BKtotcomm} are related to resolvent centrality and to subgraph centrality, respectively, so one would expect that similar results may hold. In Fig.4 we give an example of  a graph with two  cospectral nodes having different Katz centrality and different
	total communicability, showing that these measures can behave quite differently from resolvent and subgraph centrality.
	
	\begin{figure}[!t]
		\centering
		\begin{tikzpicture}[x=0.76cm, y=0.76cm]
		\tikzset{palla/.style={circle, draw=black, minimum size = 0.7cm,outer sep =0, inner sep=0}}
		\tikzmath{\l=2;};
		
		\node[palla] (1) at (2,1.5) {1};
		\node[palla] (2) at (2,-1.5) {2}; 
		\node[palla] (3) at (-5,-1.5) {3};
		\node[palla] (4) at (-5,1.5) {4}; 
		\node[palla] (5) at (5,-1.5) {5};
		\node[palla] (6) at (5,1.5) {6}; 
		\node[palla] (7) at (-3,0) {7};
		\node[palla] (8) at (0,0) {8};

		\draw (4) -- (7) -- (8) -- (1) -- (6) -- (2) -- (8);
		\draw (3) -- (7);
		\draw (1) -- (5) -- (2);
		\end{tikzpicture}	
		\caption{Vertices 1 and 8 are cospectral. Along with other three graphs with 8 vertices, it is the smallest graph we have found with a pair of cospectral vertices with different Katz centrality and total communicability.}
	\end{figure}
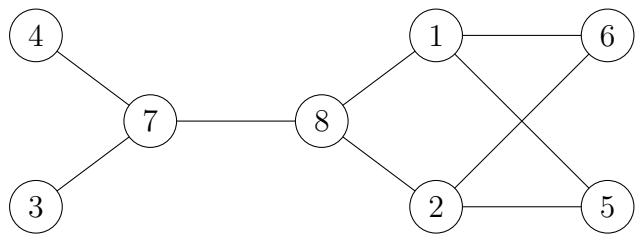

	From computer experiments, we have noticed that when two vertices exhibit interlacing with respect to resolvent centrality, they also interlace with (exponential)
	subgraph centrality. Hence, we put forth the following problem.
	
	\begin{prob}
		For a given pair of vertices of a graph $\G$, prove or disprove the following statement:
		the number of interlacing values relative to resolvent centrality is always less than or equal to the number of interlacing values relative to subgraph centrality.
	\end{prob}
	
	\vspace{0.3cm}
	\noindent {\bf Acknowledgement}: We would like to thank Sinan Aksoy and Stephen Young for inviting us to write this article as well as two anonymous referees for their comments and suggestions.
	
	%%  The bibliography

\end{document}